\documentclass[a4paper,11pt]{amsart}
\usepackage[margin= 2.5 cm,bottom=19mm]{geometry}
\usepackage{thm-restate}
\usepackage{graphicx}
\usepackage{geometry}
\usepackage{amsthm}
\usepackage{amssymb}
\usepackage{amsmath}
\usepackage{comment}
\usepackage{hyperref}
\usepackage{xcolor}
\usepackage{enumerate}
\usepackage{enumitem}
\usepackage{listings}
\usepackage{complexity}
\usepackage{blkarray}
\usepackage{lmodern}
\usepackage[english]{babel}
\usepackage{accents}
\usepackage{float}
\usepackage{tikz}
\usetikzlibrary{decorations.pathmorphing}

\tikzset{snake it/.style={decorate, decoration=snake}}

\usetikzlibrary{calc}
\usetikzlibrary{decorations.pathreplacing}
\usetikzlibrary{positioning,patterns}
\usetikzlibrary{arrows,shapes,positioning}
\usetikzlibrary{decorations.markings}

\tikzstyle{edge}=[very thick]
\definecolor{bostonuniversityred}{rgb}{0.8, 0.0, 0.0}
\definecolor{arsenic}{rgb}{0.23, 0.27, 0.29}
\tikzstyle{diredge}=[postaction={decorate,decoration={markings,
		mark=at position .95 with {\arrow[scale = 1]{stealth};}}}]

\newcommand{\fitellipsis}[2] 
{\draw [fill=gray]let \p1=(#1), \p2=(#2), \n1={atan2(\y2-\y1,\x2-\x1)}, \n2={veclen(\y2-\y1,\x2-\x1)}
    in ($ (\p1)!0.5!(\p2) $) ellipse [ x radius=\n2/2+0cm, y radius=0.4cm, rotate=\n1];
}

\usepackage[font=small,labelfont=bf]{caption}
\usepackage[font=small,labelfont=normalfont,labelformat=simple]{subcaption}
\graphicspath{{figs/}}

\newtheorem{theorem}{Theorem}[section]
\newtheorem{lemma}[theorem]{Lemma}
\newtheorem{corollary}[theorem]{Corollary}

\newtheorem{conjecture}[theorem]{Conjecture}

\newtheorem{claim}[theorem]{Claim}

\theoremstyle{definition}

\newenvironment{claimproof}[1][\proofname]{\proof[#1]}{\endproof}

\author[Martinsson]{Anders Martinsson}
\address[Martinsson]{Department of Computer Science, Institute of Theoretical Computer Science, ETH Z\"{u}rich, Switzerland}
\email{\tt {anders.martinsson}@inf.ethz.ch}
\author[Steiner]{Raphael Steiner}
\address[Steiner]{Department of Computer Science, Institute of Theoretical Computer Science, ETH Z\"{u}rich, Switzerland}
\email{\tt {raphaelmario.steiner}@inf.ethz.ch}

\thanks{The second author was supported by an ETH Z\"{u}rich Postdoctoral Fellowship.
}

\date{\today}

\title{Strengthening Hadwiger's conjecture for $4$- and $5$-chromatic graphs}
\begin{document} 
\maketitle

\begin{abstract}
Hadwiger's famous coloring conjecture states that every $t$-chromatic graph contains a $K_t$-minor. Holroyd [Bull. London Math. Soc. 29, (1997), pp. 139--144] conjectured the following strengthening of Hadwiger's conjecture: If $G$ is a $t$-chromatic graph and $S \subseteq V(G)$ takes all colors in every $t$-coloring of $G$, then $G$ contains a $K_t$-minor \emph{rooted at $S$}. 

We prove this conjecture in the first open case of $t=4$. Notably, our result also directly implies a stronger version of Hadwiger's conjecture for $5$-chromatic graphs as follows:

Every $5$-chromatic graph contains a $K_5$-minor with a singleton branch-set. In fact, in a $5$-vertex-critical graph we may specify the singleton branch-set to be any vertex of the graph.
\end{abstract}

\section{Introduction}

Given a graph $G$ and a number $t \in \mathbb{N}$, a \emph{$K_t$-minor} in $G$ is a collection $(B_i)_{i=1}^{t}$ of pairwise disjoint non-empty subsets of $V(G)$ such that $G[B_i]$ is connected for every $i \in [t]$ and for every distinct $i, j \in [t]$ the sets $B_i$ and $B_j$ are adjacent\footnote{Two subsets $A, B$ of the vertex-set $V(G)$ of a graph $G$ are called \emph{adjacent} if there exists an edge in $G$ with one endpoint in $A$ and one endpoint in $B$.} in $G$. The sets $B_1, \ldots,B_t$ are also called the \emph{branch-sets} of the $K_t$-minor. 

Hadwiger's conjecture, among the most famous and difficult open problems in graph theory, states the following. 

\begin{conjecture}[Hadwiger 1943~\cite{hadwiger}]
For every integer $t \ge 1$, every graph $G$ with $\chi(G)=t$ contains a $K_t$-minor.
\end{conjecture}

Hadwiger's conjecture was proved for $t \le 4$ by Hadwiger himself~\cite{hadwiger}, and also by Dirac~\cite{dirac}, while the case $t=5$ was shown to be equivalent to the Four-Color-Problem by Wagner~\cite{wagner}. Since Appel and Haken's proof of the Four-Color-Theorem~\cite{appelhaken1,appelhaken2} in 1976 and its later independent proof by Robertson, Sanders, Seymour and Thomas~\cite{robertson2}, Hadwiger's conjecture was known to hold for $t=5$. In a tour de force, in 1993 Robertson, Seymour and Thomas~\cite{robertson1} proved Hadwiger's conjecture for $t=6$ by reducing it to the Four-Color Theorem. All the cases $t \ge 7$ remain wide open as of today. Much of recent research has focused on asymptotic bounds for coloring graphs with no $K_t$-minor rather than on the precise version of the conjecture, and there has been some exciting progress in this direction recently~\cite{del,norin}. For further partial results on Hadwiger's conjecture and background, we refer the reader to the survey article~\cite{survey} by Seymour.  

In this paper, we investigate a strengthening of Hadwiger's conjecture proposed by Holroyd~\cite{holroyd} in 1997. His stronger version of Hadwiger's conjecture concerns the containment of minors in graphs of given chromatic number that are \emph{rooted} at a particular set of vertices.
Rooted minors have received significant attention before, we refer to~\cite{demasi,ellen,rooted,hayashi,kriesell,kriesell2,marx,wollan} for a selection of results on rooted minors in graphs. Here we use the following definition: For a set $S \subseteq V(G)$, we say that a $K_t$-minor $(B_i)_{i=1}^{t}$ in $G$ is \emph{$S$-rooted} or \emph{rooted at $S$} if $B_i \cap S \neq \emptyset$ for all $i \in [t]$. A second notion we need to introduce before stating Holroyd's conjecture is that of \emph{colorful sets}. For a graph $G$ and $S \subseteq V(G)$ we say that $S$ is \emph{colorful} in $G$ if for every proper $\chi(G)$-coloring of $G$ there are vertices of all $\chi(G)$ colors contained in $S$.

A colorful set in a graph may be seen as a part of the graph that ``is hard to color''. In that sense, it is quite intuitive that one may hope to find a $K_t$-minor rooted at this part of the graph. This is exactly Holroyd's conjecture (named ``Strong Hadwiger's conjecture'' in~\cite{holroyd}):

\begin{conjecture}[Strong Hadwiger's conjecture, Holroyd 1997~\cite{holroyd}]
Let $G$ be a graph with $\chi(G)=t$, and let $S$ be a colorful set in $G$. Then $G$ contains an $S$-rooted $K_t$-minor.
\end{conjecture}

Holroyd proved his conjecture for $t \le 3$ in~\cite{holroyd}. Here, we take the next step and prove the Strong Hadwiger's conjecture for $t=4$. 

\begin{theorem}\label{thm:main}
Let $G$ be a graph with $\chi(G)=4$ and let $S$ be a colorful set in $G$. Then $G$ contains an $S$-rooted $K_4$-minor. 
\end{theorem}

As was already noted by Holroyd (Theorem~2 in~\cite{holroyd}), the truth of the Strong Hadwiger's conjecture for a value $t$ implies the truth of Hadwiger's conjecture for the value $t+1$. Therefore, Theorem~\ref{thm:main} implies Hadwiger's conjecture for $t=5$. In fact, combining Holroyd's argument with Theorem~\ref{thm:main} we obtain a slight strengthening of Hadwiger's conjecture for $t=5$ as follows. 

\begin{corollary}\noindent
\begin{itemize}
\item Every $5$-vertex-critical graph $G$ for every $v \in V(G)$ has a $K_5$-minor containing $\{v\}$ as a singleton branch-set. 
\item Every $5$-chromatic graph contains a $K_5$-minor with a singleton branch-set.
\end{itemize}
\end{corollary}
\begin{proof}\noindent
\begin{itemize}
    \item Let $v \in V(G)$. Then by criticality of $G$ we have $\chi(G)=5, \chi(G-v)=4$. This implies that $N_G(v)$ is a colorful set in $G-v$. By Theorem~\ref{thm:main} there exists an $N_G(v)$-rooted $K_4$-minor in $G-v$, and adding to it the branch-set $\{v\}$ yields the desired $K_5$-minor.
    \item This follows from the first item since every $5$-chromatic graph has a $5$-critical subgraph. 
\end{itemize}
 
\end{proof}
\noindent To the best of our knowledge these strengthenings of Hadwiger's conjecture for $t=5$ are novel.

\medskip

\paragraph{\textbf{Organization and Outline.}} The rest of the paper is devoted to the proof of Theorem~\ref{thm:main}. In Section~\ref{sec:prelim}, we prepare the proof by discussing a few results from the literature that will be needed in our proof. Most crucially, we will rely on a characterization for the existence of $K_4$-minors rooted at a set of four distinct vertices by Fabila-Monroy and Wood~\cite{rooted}. We then present our proof in Section~\ref{sec:proof}. A main part of the proof is to establish several properties of a smallest counterexample $G$ to Theorem~\ref{thm:main} in terms of connectivity and the distribution of the vertices in $S$ over the graph. In another step, these structural properties can then either be used to find the desired rooted $K_4$-minor directly, or to restrict the structure of $G$ to a planar graph (Lemma~\ref{lma:rootyness}). At this final stage, we invoke the Four-Color-Theorem to derive a $4$-coloring that shows that $S$ cannot be colorful and this yields the desired contradiction.

\section{Preliminaries}\label{sec:prelim}
In this section we state a few preliminary results from the literature that we will use in our proof of Theorem~\ref{thm:main}. The first is a precise characterization when a graph admits a $K_3$-minor rooted at $3$ given vertices due to Holroyd~\cite{holroyd} and Linusson and Wood~\cite{linusson}.
\begin{lemma}[Theorem~6 in~\cite{holroyd}, Lemma~5 in~\cite{linusson}]~\label{lemma:K3minor}
Let $G$ be a graph, and let $a,b,c \in V(G)$ be three distinct vertices. If for every vertex $v \in V(G)$ at least two of $a,b,c$ are in a common component of $G-v$, then there exists an $\{a,b,c\}$-rooted $K_3$-minor in $G$.
\end{lemma}

Second, we will need a result by Fabila-Monroy and Wood \cite{rooted} for characterising when a graph contains a $K_4$-minor rooted at four given verties $a,b,c,d$. Following the notation in~\cite{rooted}, given a graph $H$, we denote by $H^+$ a graph obtained from $H$ by, for each triangle $T$ in $H$, adding a disjoint (possibly empty) set of vertices $C_T$ to $H$, making them a clique and adjacent to the three vertices of the triangle $T$ in $H$ (the sets $C_T$ and $C_{T'}$ must be chosen disjoint for different triangles $T, T'$ in $H$). An $\{a, b, c, d\}$-web is a graph $H^+$ where $H$ has a planar embedding with outerface $\{a, b, c, d\}$ in some order, and such that every triangle in $H$ forms a face in this embedding. Fabila-Monroy and Wood~\cite{rooted} give a precise description of the edge-maximal graphs with no $K_4$-minor rooted at four specified vertices. Here, we will only need a weaker version of their main result.

\begin{theorem}\label{theorem:wood}
Let $G$ be a graph with four marked vertices $a,b,c,d$ such that $G$ contains no $\{a,b,c,d\}$-rooted $K_4$-minor. Then either
\begin{enumerate}
\item $G$ is a spanning subgraph of an $\{a,b,c,d\}$-web, or
\item there exist vertices $u,v \in V(G)$ such that $G-\{u,v\}$ has at least three distinct connected components.
\end{enumerate}
\end{theorem}
\begin{proof}
By Theorem 15 in \cite{rooted}, $G$ is a spanning subgraph of a graph in one the classes $\mathcal{A}$-$\mathcal{F}$. Class $\mathcal{D}$ is precisely the $\{a,b,c,d\}$-webs in which case $(1)$ holds. It can be directly checked that all graphs in the remaining classes (and thus also their spanning subgraphs) have a $2$-separator as described in $(2)$.
\end{proof}

Throughout the proof we will use the following convenient (and standard) terminology to speak about separators in graphs:
A \emph{separation} of a graph is a tuple $(A,B)$ where $A, B$ are subsets of $V(G)$ with $A \cup B=V(G)$, such that $A \setminus B\neq \emptyset \neq B \setminus A$ and such that there are no edges in $G$ connecting a vertex in $A\setminus B$ to a vertex in $B \setminus A$. The \emph{order} of a separation is $|A \cap B|$, and $A\cap B$ is called the \emph{separator} of the separation $(A,B)$. It is easy to see that a graph is $k$-connected if and only if all its separations have order at least $k$. In the proof, we repeatedly use Menger's theorem for vertex-disjoint paths in the following two variants:

\begin{theorem}[Menger's theorem-set versions, cf. Theorem~3.3 in~\cite{diestel}]
Let $G$ be $k$-connected, $v \in V(G)$ and $A, B \subseteq V(G)$. 
\begin{itemize}
    \item There are $\min\{k,|A|,|B|\}$ vertex-disjoint paths in $G$ each with endpoints in $A$ and $B$. 
\item If $v\notin A$ then there are $\min\{k,|A|\}$ paths in $G$, each of them connecting $v$ to a vertex in $A$, and any two of them only sharing the vertex $v$.
\end{itemize}
\end{theorem}

We note that in the above statement as well as throughout the whole manuscript we allow paths to consist of single vertices (i.e., be of length $0$). 

\section{The proof}\label{sec:proof}

For a $3$-connected graph $G$ and a subset $S \subseteq V(G)$ of vertices, we say that $S$ is \emph{spread out in $G$} if for every separation $(A,B)$ in $G$ of order $3$, it holds that $S \setminus A \neq \emptyset \neq S \setminus B$. The following auxiliary statement will be of crucical use in our proof of Theorem~\ref{thm:main}. 

\begin{lemma}\label{lma:rootyness}
Let $G$ be a $3$-connected graph, and let $S \subseteq V(G)$ such that $|S|\ge 4$ and $S$ is spread out in $G$. If $G$ contains no $S$-rooted $K_4$-minor then the graph $G^a(S)$, obtained from $G$ by adding a new vertex and making it adjacent to every element of $S$, is planar.
\end{lemma}

We will prove Lemma \ref{lma:rootyness} by combining Theorem \ref{theorem:wood} with the following statement.
\begin{lemma}\label{lma:induction}
Let $G$ be a 3-connected graph, and let $S\subseteq V(G)$ be spread out in $G$ such that $G$ contains no $S$-rooted $K_4$-minor. Then for any separation $(A, B)$ of order $3$ such that $|A\cap S|\geq 3$, it holds that $G[B]$ has a planar embedding with $A\cap B$ on the outerface. 
\end{lemma}

\begin{proof}[Proof of Lemma~\ref{lma:induction}]

We prove Lemma \ref{lma:induction} by induction on the size of $B$. Note that for any separation, we have $|B|\geq 4$. Moreover, for $|B|=4$, it is clear that such a planar embedding exists. Turning to the induction step, we fix any separation $(A, B)$ as above with $|B|\geq 5$, and assume the statement is true for any separation $(A', B')$ as above with $|B'|<|B|$. Furthermore, we pick some $S' \in \binom{A \cap S}{3}$.

\begin{claim}\label{claim:connected} $G[A]$ and $G[B]$ are connected.
\end{claim}
\begin{claimproof}

By Menger's theorem, any $v\in A\setminus B$  is connected to $A\cap B$ in $G$ by three paths that pairwise only share $v$ as a common vertex. These paths are completely contained in $A$ as $A\cap B$ is a separator, so, in particular, $G[A]$ is connected. Similarly, any $v\in B\setminus A$ is connected to $A\cap B$ by three internally disjoint paths in $G$, so $G[B]$ is also connected. 
\end{claimproof}
\begin{claim}\label{claim:cherryminor}
There exist enumerations $\{s_1, s_2, s_3\}=S'$ and $\{c_1, c_2, c_3\} = A\cap B$ and disjoint sets $A_1, A_2, A_3 \subseteq A$ such that $\{s_i, c_i\}\subseteq A_i$ and $G[A_i]$ is connected for each $i \in [3]$ and such that $A_2$ is adjacent to $A_1$ and $A_3$ in $G$.
\end{claim}
\begin{claimproof}
Since $G$ is $3$-connected, there exists three vertex-disjoint paths from $S'$ to $A\cap B$. Let $P_1, P_2, P_3$ be an ordering of the paths such that $\text{dist}_{G[A]}(V(P_1), V(P_2))+\text{dist}_{G[A]}(V(P_2), V(P_3))$ is minimized. Then letting $A_1 = V(P_1)$, $A_3=V(P_3)$ and $A_2$ equal the connected component of $G[A\setminus (V(P_1) \cup V(P_3))]$ containing $V(P_2)$, we obtain sets with the desired properties (by the minimality assumption a shortest path from $V(P_2)$ to $V(P_i)$ for $i \in \{1,3\}$ does not intersect $V(P_{4-i})$, certifying that $A_2$ is adjacent to $A_i$).
\end{claimproof}

We will below let $s_1, s_2, s_3$ and $c_1, c_2, c_3$ denote the orderings of the respective sets, as prescribed by Claim \ref{claim:cherryminor}.

\begin{claim}\label{claim:cspread}
Let $X\subseteq B$ be such that $|X| \in \{1,2\}$. Any connected component of $G[B\setminus X]$ contains at least one of the vertices $c_1, c_2, c_3$.
\end{claim}
\begin{claimproof}
$G-X$ is connected by $3$-connectivity of $G$. Thus for any $v\in B\setminus X$ there exists a shortest path from $v$ to $(A\cap B)\setminus X$ in $G-X$. By minimality, this path is contained in $B$. Hence $v$ is in the same connected component of $G[B\setminus X]$ as the end-point of this path.
\end{claimproof}

\begin{claim}\label{claim:cutvertex} If $G[B]$ has a cut-vertex, then $G[B]$ has a planar embedding with $A\cap B$ on its outerface.
\end{claim}
\begin{claimproof}

Let $x$ be a cut-vertex in $G[B]$. By Claim \ref{claim:cspread}, one of $c_1, c_2, c_3$ has to be in a connected component of $G[B\setminus \{x\}]$ that does not contain the other two vertices. We may w.l.o.g. assume that $c_1$ has this property. As each vertex in $G[B\setminus\{x,c_1\}]$ is in the same connected component as $c_2$ or $c_3$, it follows that $c_1$ is isolated in $G[B\setminus\{x\}]$, hence $x$ is the only neighbor of $c_1$ in $G[B]$.

It is straight-forward to see that $(A\cup \{c_1\}, B\setminus\{c_1\})$ is a separation in $G$ with separator $\{x, c_2, c_3\}$. By the induction hypothesis, $G[B\setminus\{c_1\}]$ has a planar embedding with $x, c_2, c_3$ on the outerface, which by adding $c_1$ to the outerface of the embedding and connecting it to $x$ can be extended to a planar embedding of $G[B]$ with $c_1,c_2,c_3$ on the outerface.

\end{claimproof}

It remains to consider the case where $G[B]$ is 2-connected and $|B|\geq 5$. For this we will use Theorem \ref{theorem:wood}.

Let $s_4 \in S\cap (B\setminus A)$ (such a vertex exists since $S$ is spread out in $G$). Observe that $G[B]$ has no $K_4$-minor rooted at $\{c_1, c_2, c_3, s_4\}$ as then using Claim \ref{claim:cherryminor} we could immediately extend it to a $K_4$-minor rooted at $S$ (replace any occurrence of a vertex $c_i$ in a branch-set by the connected set $A_i \supseteq \{s_i,c_i\}$ in $A$). Hence, $G[B]$ must satisfy either Cases $(1)$ or $(2)$ of Theorem \ref{theorem:wood}.

Let us first consider Case $(1)$. Then $G[B]$ is a spanning subgraph of a $\{c_1,c_2,c_3,s_4\}$-web $H^+$. Concretely, $H$ is a planar graph that admits a planar embedding with outerface $\{c_1,c_2,c_3,s_4\}$ (in some possibly permuted order) and such that every triangle $T$ in $H$ forms a face in this embedding. Further, for every triangle $T$ in $H$ there is a disjoint (possibly empty) set $C_T$ of vertices in $H^+$ such that $H^+$ is formed from $H$ by making $C_T$ a clique fully connected to the three vertices of $T$ for every triangle $T$.

Now note that for every triangle $T$ in $H$ for which $C_T \neq \emptyset$, we have that $(V(H),V(T) \cup C_T)$ forms a separation of order $3$ in $G[B]$ with separator $V(T)$. Further, $(A \cup V(H), V(T) \cup C_T)$ is a separation of order $3$ in the whole graph $G$. Since $H$ contains at least $4$ vertices, we have $|V(T) \cup C_T|<|B|$ for every triangle $T$ in $H$ and now the inductive assumption implies that $G[V(T) \cup C_T]$ admits a planar embedding with $V(T)$ on its outerface, for every triangle $T$ in $H$. It is now easy to see that we can take the planar embedding of $H$ as prescribed above, and for each triangular face $T$ with $C_T \neq \emptyset$ glue into it an embedding of $G[V(T) \cup C_T]$ with $V(T)$ on its outerface, to overall reach a planar embedding with $c_1,c_2,c_3$ (and in fact also $s_4$) on its outerface. After deleting some edges from this embedding we reach a planar embedding of $G[B]$ with $c_1,c_2,c_3$ on the outerface. This is the desired outcome and concludes the proof in this case.

It only remains to consider Case $(2)$. By Claims \ref{claim:connected} and \ref{claim:cutvertex} we may assume that $G[B]$ is $2$-connected. Let $\{x, y\} \subseteq B$ be the prescribed cut such that $G[B\setminus \{x,y\}]$ has at least $3$ distinct connected components. As by Claim \ref{claim:cspread}, any connected component of $G[B\setminus\{x,y\}]$ contains a vertex among $c_1, c_2, c_3$, we conclude that $G[B\setminus\{x,y\}]$ has exactly three connected components $C_1, C_2, C_3$ where we may w.l.o.g. assume $c_i\in C_i$ for $i=1, 2, 3$.

To conclude the induction proof, we will show that these assumptions are sufficient to find branch sets $c_1\in B_1, c_2\in B_2, c_3\in B_3, s_4\in B_4 \subseteq B$ such that any pair of sets except for possibly $B_2$ and $B_3$ are adjacent. Combining this with Claim \ref{claim:cherryminor}, this would imply the existence of an $S$-rooted $K_4$-minor (with branch-sets $A_1 \cup B_1, A_2 \cup B_2, A_3 \cup B_3$ and $B_4$), hence concluding the proof. 

As $G[B]$ is $2$-connected, there exist for each $i=1, 2, 3$ two internally disjoint paths $P_i, Q_i$ in $G[B]$ with endpoints $c_i, x$ and $c_i, y$, respectively. Then as $\{x, y\}$ separates $C_i$ from the rest of $B$, we must have $V(P_i), V(Q_i) \subseteq C_i\cup \{x, y\}$. Let $X:=\bigcup_{i=1}^{3}{(V(P_i)\setminus \{c_i\})}$ and $Y:=\bigcup_{i=1}^{3}{(V(Q_i)\setminus \{c_i\})}$. Note that $G[X]$ and $G[Y]$ are connected. If $s_4$ is already contained in either $X$ or $Y$, let us w.l.o.g. assume $X$, then $B_1=\{c_1\}\cup Y, B_2 = \{c_2\}, B_3=\{c_3\}$ and $B_4=X$ are vertex sets with the desired properties. Otherwise $s_4$ is contained in some component $C_i$. Note that by choice of $s_4$ it is distinct from $c_i$. Let $P$ be a shortest path from $s_4$ to $X\cup Y$ in $G[B\setminus\{c_i\}]$. Such a path exists as $G[B]$ is $2$-connected. Let us, w.l.o.g., assume that the end-point of $P$ is in $X$. Then similar to before $B_1=\{c_1\}\cup Y, B_2 = \{c_2\}, B_3=\{c_3\}$ and $B_4=X\cup V(P)$ have the desired properties. This concludes the proof of Lemma \ref{lma:induction}. 
\end{proof}


Before giving the proof of Lemma~\ref{lma:rootyness} it will be useful to observe the following simple statement about planar graphs. 

\begin{lemma}\label{lemma:kuratowski}
Let $G$ be a planar graph and let $S \subseteq V(G)$ be a set with $|S|\ge 4$. Suppose that for every choice of vertices $s_1,s_2,s_3,s_4 \in S$ there is a planar embedding of $G$ in which $s_1,\ldots,s_4$ are on the boundary of the outerface. Then the graph $G^a(S)$, obtained from $G$ by adding a new vertex whose neighbors are the elements of $S$, is planar.
\end{lemma}
\begin{proof}
We claim that for every 4-subset $\{s_1,s_2,s_3,s_4\} \subseteq S$, the graph $G^a(s_1,s_2,s_3,s_4)$, obtained by adding a new vertex to $G$ with neighborhood $\{s_1,s_2,s_3,s_4\}$, is planar. Indeed, given a planar embedding with $s_1,s_2,s_3,s_4$ on the outerface, we can simply embed the new vertex in the outerface and preserve a planar embedding. Now, suppose towards a contradiction that $G^a(S)$ is non-planar. By Kuratowksi's theorem, this means that $G^a(S)$ contains a subdivision of $K_5$ or of $K_{3,3}$. Fix one such subdivision of a Kuratowski graph. It is a subgraph of $G^a(S)$ of maximum degree at most $4$. But then it has to be also contained in a graph $G^a(s_1,s_2,s_3,s_4)$ for some $\{s_1,s_2,s_3,s_4\} \subseteq S$, a contradiction, since the latter graph is planar.
\end{proof}

\begin{proof}[Proof of Lemma \ref{lma:rootyness}.]
Take any four vertices $s_1, s_2, s_3, s_4\in S.$ By Theorem \ref{theorem:wood} it follows that if there is no $K_4$-minor rooted at these four vertices, then as $G$ is $3$-connected it must satisfy Case $(2)$ with $s_1, s_2, s_3, s_4$ on the outerface. Concretely, $G$ is a spanning subgraph of an $\{s_1,s_2,s_3,s_4\}$-web $H^+$. This means $H$ is a graph with a planar embedding with outerface $\{s_1,s_2,s_3,s_4\}$ (in some order) and every triangle $T$ in $H$ forms a face in this embedding. Further, for every triangle $T$ in $H$ there is a disjoint set $C_T$ of vertices in $H^+$ such that $H^+$ is formed from $H$ by making $C_T$ a clique fully connected to the three vertices of $T$ for every triangle $T$. Since for every triangle $T$ in $H$ we have that $(V(H),V(T) \cup C_T)$ forms a $3$-separation in $G$, by applying Lemma \ref{lma:induction} we find for every triangle $T$ in $H$ a planar embedding of $G[V(T)\cup C_T]$ with $V(T)$ on the outerface. Hence, by taking the prescribed planar embedding of $H$ and glueing for every triangle $T$ a planar embedding of $G[V(T) \cup C_T]$ as above into the corresponding face, we find a planar embedding of a supergraph of $G$ that has $s_1, s_2, s_3, s_4$ on the outerface. This yields a planar embedding also of $G$ with $s_1,\ldots,s_4$ on the outerface.

Since our choice of $s_1,s_2,s_3,s_4$ was arbitrary, Lemma~\ref{lemma:kuratowski} implies that $G^a(S)$ is planar. 
\end{proof}

Equipped with Lemma~\ref{lma:rootyness} we can now present the proof of our main result.

\begin{proof}[Proof of Theorem~\ref{thm:main}]
We prove the theorem by contradiction. Suppose (reductio ad absurdum) that a graph $G$ is a smallest counterexample (in terms of $|V(G)|$). Then $\chi(G)=4$, and there is a colorful set $S$ of vertices in $G$ such that $G$ contains no $S$-rooted $K_4$-minor. We start by establishing some simple facts concerning the connectivity of $G$. 
\begin{claim}
$G$ is connected. 
\end{claim}
\begin{claimproof}
Suppose not, then there exists a partition $(A,B)$ of $V(G)$ with $A, B \neq \emptyset$ such that no edge in $G$ connects $A$ and $B$. Then $\chi(G[A]), \chi(G[B]) \le 4$ and $G[A]$ and $G[B]$ respectively do not contain a $K_4$-minor rooted at $S \cap A$ and $S \cap B$ respectively. By minimality of $G$, there exist proper $4$-colorings $c_A:A \rightarrow [4]$ and $c_B:B \rightarrow [4]$ such that $S \cap A$ does not contain all $4$ colors in $c_A$, and $S \cap B$ does not contain all $4$ colors in $c_B$. By permuting colors we may assume w.l.o.g. that $c_A(S \cap A) \subseteq \{2,3,4\}$ and $c_B(S \cap B) \subseteq \{2,3,4\}$. But then putting together $c_A$ and $c_B$ yields a proper coloring of $G$ in which no vertex in $S$ receives color $1$, a contradiction to $S$ being colorful. 
\end{claimproof}
\begin{claim}
$G$ is $2$-connected. 
\end{claim}
\begin{claimproof}
Towards a contradiction, suppose there exists a separation $(A,B)$ of $G$ of order $|A \cap B|=1$. Write $A \cap B=\{v\}$. Clearly, $\chi(G[A]), \chi(G[B]) \le 4$. 

Suppose first that $S \cap A =\emptyset$ or $S \cap B=\emptyset$. By symmetry, we may assume w.l.o.g. that the latter holds, i.e. $S \subseteq A\setminus B$. Then $G[A]$ contains no $S$-rooted $K_4$-minor, hence there is a proper $4$-coloring $c_A:A \rightarrow [4]$ of $G[A]$ in which no vertex in $S$ receives color $1$. Let $c_B:B \rightarrow [4]$ be a proper $4$-coloring of $G$. Possibly after permuting colors in $c_B$ we have $c_A(v)=c_B(v)$, and then the common extension of $c_A$ and $c_B$ to $G$ forms a proper $4$-coloring in which no vertex in $S$ receives color $1$, a contradiction.

Moving, suppose that $S \cap A \neq\emptyset \neq  S \cap B$. We claim that $G[A]$ does not contain an $(S \cap A) \cup \{v\}$-rooted $K_4$-minor. Indeed, if $G[A \cup \{v\}]$ contains a $K_4$-minor $(B_i)_{i=1}^{4}$ rooted at $(S \cap A) \cup \{v\}$, then one of the branch-sets must contain $v$ (for otherwise $(B_i)_{i=1}^{4}$ is an $S$-rooted $K_4$-minor in $G$), say $v \in B_1$. By connectivity, there exists a path $P$ in $G$ connecting $v$ to a vertex in $S \cap B$, and $V(P) \subseteq B$. Then $(B_1 \cup V(P), B_2, B_3, B_4)$ form the branch-set of an $S$-rooted $K_4$-minor in $G$, a contradiction. The symmetric argument works to show that $G[B]$ contains no $(S \cap B) \cup \{v\}$-rooted $K_4$-minor.

By minimality of $G$, it follows that there exist proper $4$-colorings $c_A:A \rightarrow [4]$ of $G_A$ and $c_B:B \rightarrow [4]$ of $G_B$ such that $(S \cap A) \cup \{v\}$ and $(S \cap B) \cup \{v\}$ do not contain a vertex of color $1$. Then $c_A(v) \neq 1 \neq c_B(v)$, and by applying a suitable $1$-invariant permutation of the colors in $c_B$ we may assume w.l.o.g. that $c_A(v)=c_B(v)$. Now, the common extension of $c_A$ and $c_B$ is a proper coloring of $G$ in which no vertex in $S$ receives color $1$, a contradiction.
\end{claimproof}
\begin{claim}
$G$ is $3$-connected. 
\end{claim}
\begin{claimproof}
Suppose not, then there exists a separation $(A,B)$ of $G$ of order $|A \cap B|=2$. Write $A \cap B=\{u,v\}$. Let us fix (arbitrarily) a proper coloring $c:V(G)\rightarrow [4]$ of $G$. We now consider two cases depending on the distribution of $S$ on the two sides $A$ and $B$ of the separation. 

\medskip
\textbf{Case 1.} $|S \cap A| \ge 2$ and $|S \cap B| \ge 2$. Since $G$ is $2$-connected, by Menger's theorem there exist two vertex-disjoint paths $P_u^A, P_v^A$ in $G$ starting in $S \cap A$ and ending at $u, v$ respectively. Similarly, there exist two disjoint paths $P_u^B, P_v^B$ starting at $S \cap B$ and ending at $u$ and $v$ respectively. Clearly, the paths $P_u^A, P_v^A$ are fully contained in $G[A]$ and the paths $P_u^B, P_v^B$ are fully contained in $G[B]$. Since $G$ is $2$-connected, it is easy to see that $G[A]$ and $G[B]$ are connected graphs, for otherwise $u$ or $v$ would form a cut-vertex of $G$. This implies that there exists a path $P^A$ in $G[A]$ and a path $P^B$ in $G[B]$ such that $P^A$ has endpoints $p_u^A \in V(P_u^A)$ and $p_v^A \in V(P_v^A)$, while $P^B$ has endpoints $p_u^B \in V(P_u^B)$ and $p_v^B \in V(P_v^B)$. W.l.o.g. (by choosing shortest paths) we may assume $V(P_u^A) \cap V(P^A)=\{p_u^A\}, V(P_v^A) \cap V(P^A)=\{p_v^A\}$ and $V(P_u^B) \cap V(P^B)=\{p_u^B\}, V(P_v^B) \cap V(P^B)=\{p_v^B\}$. It is then easy to see that the sets $(V(P_u^A), V(P_v^A) \cup (V(P^A)\setminus\{p_u^A\}))$ form a $K_2$-minor in $G[A]$ rooted at $u,v$ and at $S$, and similarly $(V(P_u^B), V(P_v^B) \cup (V(P^B)\setminus\{p_u^B\}))$ form a $K_2$-minor in $G[B]$ rooted at $u,v$ and at $S$.

Consider first the case that $c(u) \neq c(v)$. Then we define $G_A, G_B$ as the graphs obtained from $G[A]$ and $G[B]$ respectively by adding an edge between $u$ and $v$ (if it does not already exist). Since $c$ induces proper colorings on $G_A$ and $G_B$, we have $\chi(G_A), \chi(G_B) \le 4$. Let us further define $S_A:=(S \cap A) \cup \{u,v\}, S_B:=(S \cap B) \cup \{u,v\}$. We claim that $G_A$ contains no $S_A$-rooted $K_4$-minor. Indeed, if $G_A$ were to contain an $S_A$-rooted $K_4$-minor, then replacing every occurrence of $u$ or $v$ in one of its branch-sets by the sets $V(P_u^B)$ or $V(P_v^B) \cup (V(P^B)\setminus\{p_u^B\})$ in $G$ respectively would yield an $S$-rooted $K_4$-minor in $G$, a contradiction. A symmetric argument shows that $G_B$ contains no $S_B$-rooted $K_4$-minor.

The above facts, $|V(G_A)|, |V(G_B)|<|V(G)|$ and the minimality assumption on $G$ now imply that there exist proper colorings $c_A:V(G_A) \rightarrow [4]$ of $G_A$ and $c_B:V(G_B)\rightarrow [4]$ of $G_B$ such that no vertex in $S_A$ receives color $1$ in $G_A$ and no vertex in $S_B$ receives color $1$ in $G_B$. In particular, $c_A(u), c_A(v)$ are distinct and not equal to $1$, and the same is true for $c_B(u), c_B(v)$. Hence, after applying a $1$-invariant permutation of $[4]$ that maps $c_B(u)$ to $c_A(u)$ and $c_B(v)$ to $c_A(v)$, we may assume w.l.o.g. that $c_A(u)=c_B(u), c_A(v)=c_B(v)$. Now, the common extension of $c_A$ and $c_B$ to $V(G)$ forms a proper coloring of $G$ in which no vertex in $S$ receives color $1$. This is a contradiction to $S$ being colorful and shows that the case $c(u) \neq c(v)$ cannot occur. 

Next, suppose that $c(u)=c(v)$. Then we define $G_A$ and $G_B$ as the graphs obtained from $G[A]$ and $G[B]$ by identifying $u$ and $v$ into a new vertex $x_{uv}$, and define $S_A:=S \cup \{x_{uv}\}$ and $S_B:=S \cup \{x_{uv}\}$. Similar to the previous case we claim that $G_A$ contains no $S_A$-rooted $K_4$-minor. Indeed, if it did, then possibly after replacing the occurence of the vertex $x_{uv}$ in one of the four branch-sets with the set $V(P_u^B) \cup V(P^B) \cup V(P_v^B)$ in $G$ yields the four branch-sets of an $S$-rooted $K_4$-minor in $G$, a contradiction. A symmetric argument shows that $G_B$ contains no $S_B$-rooted $K_4$-minor. Note that $c$ induces a proper $4$-coloring on both $G_A$ and $G_B$, so that $\chi(G_A), \chi(G_B) \le 4$. These facts imply (using the minimality of $G$ and that $G_A$ and $G_B$ are smaller than $G$) the existence of proper $[4]$-colorings $c_A, c_B$of $G_A$ and $G_B$ respectively in which no vertex in $S_A$ and $S_B$ respectively gets assigned color $1$. In particular, this means $c_A(x_{uv}) \neq 1 \neq c_B(x_{uv})$. Hence (possibly after applying a $1$-invariant permutation of the color-set $[4]$ in $c_B$ that maps $c_B(x_{uv})$ to $c_A(x_{uv})$) we may assume w.l.o.g. that $c_A(x_{uv})=c_B(x_{uv})$. It is now easy to see that $c^\ast:V(G)\rightarrow [4]$ defined by $c^\ast(u):=c^\ast(v):=c_A(x_{uv})$, $c^\ast(x):=c_A(x)$ for $x \in A \setminus \{u,v\}$, $c^\ast(x):=c_B(x)$ for $x \in B \setminus \{u,v\}$, forms a proper $4$-coloring of $G$ in which no vertex in $S$ gets assigned color $1$. This yields a contradiction to the fact that $S$ is colorful and hence shows that also the case $c(u)=c(v)$ is impossible. All in all, we conclude that Case~1 cannot occur.

\medskip
\textbf{Case 2.} $|S \cap A|\le 1$ or $|S \cap B| \le 1$. W.l.o.g. we assume $|S \cap B|\le 1$ in the following. We denote by $P$ a path in $G[B]$ connecting $u$ and $v$ (such a path must exist, since $G$ is $2$-connected). Furthermore, if $S \setminus A\neq \emptyset$, that is, $S \cap B=\{s\}$ for some $S \in B \setminus A$, then by $2$-connectivity of $G$ and Menger's theorem there exist two internally disjoint paths in $G$ connecting $s$ to $u$ and $v$, which we denote by $P_u$ and $P_v$ respectively. Clearly, $P_u$ and $P_v$ have to be entirely contained in $G[B]$. 

Moving on, let us first consider the case $c(u)=c(v)$. Let $G_A$ be the graph obtained from $G[A]$ by identifying $u$ and $v$ into a single vertex $x_{uv}$. Note that $c$ induces a proper $4$-coloring on $G_A$, and so $\chi(G_A) \le 4$. We furthermore define a set $S_A$ of vertices in $G_A$ as $S_A:=S$ if $S \cap B=\emptyset$ and $S_A:=S \cup \{x_{uv}\}$ if $S \cap B \neq \emptyset$. We claim that $G_A$ does not contain an $S_A$-rooted $K_4$-minor. In the case $S \subseteq A$, this follows since the occurrence of $x_{uv}$ in any branch-set of an $S_A$-rooted $K_4$-minor in $G_A$ can be replaced by the full set $V(P)$ in $G$, which would result in an $S$-rooted $K_4$-minor in $G$, a contradiction. Similarly, in the case $S \setminus A \neq \emptyset$ the occurrence of $x_{uv}$ in a branch-set of an $S_A$-rooted $K_4$-minor in $G_A$ can be replaced by the set $V(P_1) \cup V(P_2)$ in $G$, again resulting in an $S$-rooted $K_4$-minor in $G$, a contradiction. Hence, $G_A$ is $4$-colorable, contains no $S_A$-rooted $K_4$-minor and is smaller than $G$. By minimality of $G$ there must exist a proper coloring $c_A:V(G_A) \rightarrow [4]$ in which no vertex in $S_A$ receives color $1$. Let $\pi \in S_4$ be a permutation with $\pi(c(u))=\pi(c(v))=c_A(x_{uv})$ and such that $\pi(s) \neq 1$ in the case that $S \setminus A \neq \emptyset$. The existence of such a $\pi$ in can be seen as follows: Either we have $c(s)=c(u)=c(v)$, in which any choice of $\pi$ with $\pi(c(u))=\pi(c(v))=c_A(x_{uv})$ automatically satisfies $\pi(c(s))=c_A(x_{uv})\neq 1$, since $x_{uv} \in S_A$ by definition. And otherwise if $c(s) \neq c(u)=c(v)$, we have enough freedom to choose $\pi$ such that it maps $c(u)=c(v)$ to $c_A(x_{uv})$ and $c(s)$ to an element distinct from $1$. 
Having found the permutation $\pi$, we can see that the coloring $c^\ast$ of $G$, defined by $c^\ast(x):=c_A(x)$ if $x \in A \setminus B$, and $c^\ast(x):=\pi(c(x))$ if $x \in B$, forms a proper coloring of $G$ in which no vertex in $S$ receives color $1$. This is a contradiction to $S$ being colorful and shows that the subcase $c(u)=c(v)$ cannot occur. 

Next, let us consider the case $c(u) \neq c(v)$. We let $G_A$ be the graph obtained from $G[A]$ by adding an edge between $u$ and $v$ (if it not already exists). Given our assumption on $c$, we have $\chi(G_A) \le 4$. We define a set $S_A\subseteq V(G)$ as follows: If $S \subseteq A$, then $S_A:=S$. If $S \setminus A \neq \emptyset$, then we define $S_A:=S \setminus \{s\}$ if $c(s) \notin \{c(u),c(v)\}$, $S_A:=(S \setminus \{s\}) \cup \{u\}$ if $c(s)=c(u)$ and $S_A:=(S \setminus \{s\}) \cup \{v\}$ if $c(s)=c(v)$. We claim that in each case, $G_A$ does not contain an $S_A$-rooted $K_4$-minor. Suppose towards a contradiction such a rooted minor $(B_i)_{i=1}^4$ exists in $G_A$. If $S \subseteq A$ or $S \setminus A \neq \emptyset$ and $c(s) \notin \{c(u),c(v)\}$, then $S_A \subseteq S$, so that we can replace any branch-set $B_i$ containing $u$ by the set $B_i \cup (V(P)\setminus \{v\})$ to obtain an $S$-rooted $K_4$-minor in $G$. 
If $S \setminus A \neq \emptyset$ and $c(s)=c(u)$, we replace any branch-set $B_i$ containing $u$ by $B_i \cup (V(P_1) \cup V(P_2))\setminus \{v\}$ to obtain an $S$-rooted $K_4$-minor in $G$. Finally, if $S \setminus A \neq \emptyset$ and $c(s)=c(v)$, we replace any branch-set $B_i$ containing $v$ by $B_i \cup ((V(P_1) \cup V(P_2))\setminus \{u\})$ to obtain an $S$-rooted $K_4$-minor in $G$.
As an $S$-rooted $K_4$-minor in $G$ does not exist by initial assumption, we find that there is indeed no $S_A$-rooted $K_4$-minor in $G_A$. Using the minimality of $G$ we conclude the existence of a proper $4$-coloring $c_A:V(G_A) \rightarrow [4]$ of $G_A$ in which no vertex in $S_A$ receives color $1$. Note that $c_A(u)\neq c_A(v)$. Let $\pi:[4]\rightarrow [4]$ be a permutation such that $\pi(c(u))=c_A(u), \pi(c(v))=c_A(v)$ and $\pi(c(s)) \neq 1$ in the case $S\setminus A \neq \emptyset$. The existence of $\pi$ can be easily seen as follows: If $S \subseteq A$, this is obvious, as we just have to map two distinct colors to a set of two other distinct colors. If $S \setminus A \neq \emptyset$ and $c(s) \notin \{c(u),c(v)\}$, we can use the range of $4$ colors to map $c(u)$ to $c_A(u)$, $c(v)$ to $c_A(v)$ and $c(s)$ to an element of $[4]\setminus\{1,c_A(u),c_A(v)\}$. If $S\setminus A \neq \emptyset$ and $c(s)=c(u)$, then $u \in S_A$ ensures that $\pi(c(s))=\pi(c(u))=c_A(u)\neq 1$ for any $\pi$ that maps $c(u)$ to $c_A(u)$ and $c(v)$ to $c_A(v)$. A symmetric argument works in the case $c(s)=c(v)$. 

Having found the permutation $\pi$, consider the $[4]$-coloring $c^\ast$ of $G$ defined as $c^\ast(x):=c_A(x)$ if $x \in A \setminus B$, and $c^\ast(x):=\pi(c(x))$ if $x \in B$. This coloring is proper and no vertex in $S$ is assigned color $1$. Again, this contradicts $S$ being colorful and shows that also the assumption $c(u)\neq c(v)$ leads to a contradiction. 

Since both assumptions on the colors of $u$ and $v$ eventually yield a contradiction, we conclude that Case~2 cannot occur.

\medskip
We found that neither Case~1 nor Case~2 is feasible. Thus our initial assumption on the existence of a separation of order $2$ in $G$ must have been wrong. $G$ is a $3$-connected graph. 
\end{claimproof}

\begin{claim}
$S$ is spread out in $G$. 
\end{claim}
\begin{claimproof}
Towards a contradiction, suppose there exists a separation $(A,B)$ in $G$ of order $3$ such that $S \subseteq A$.

Suppose first that $|B \setminus A|=1$. Writing $B \setminus A=\{x\}$, we can see that all neighbors of the vertex $x$ in $G$ are members of the separator $A \cap B$ of size $3$. Hence $x$ has degree at most $3$ in $G$. By initial assumption on $(A,B)$, we have $x \notin S$. Now, consider the graph $G-x$. By minimality of $G$, there exists a proper $[4]$-coloring of $G-x$ in which no vertex in $S$ receives color $1$. We may extend this to a proper coloring of $G$ by picking a color for $x$ that does not appear on its neighborhood. In this coloring of $G$ still no vertex in $S$ receives color $1$, contradicting that $S$ is colorful. 

Next, suppose $|B \setminus A|\ge 2$. Write $\{d_1,d_2,d_3\}=A \cap B$ for the vertices in the separator induced by $(A,B)$. Note that since $G$ is $3$-connected, by Menger's theorem for every $b \in B \setminus A$ there exist $3$ internally disjoint paths connecting $b$ to $d_1, d_2$ and $d_3$ respectively. Clearly, these paths have to be fully contained in $G[B]$. We now claim that for every vertex $v \in B$, at least two of $d_1,d_2,d_3$ are contained together in a common connected component of $G[B]-v$. Indeed, pick a vertex $b \in (B\setminus A)\setminus\{v\}$. Since $b$ is connected to $d_1,d_2,d_3$ by three internally disjoint paths in $G[B]$, at least two of these paths still exist in $G[B]-v$ and thus certify that their endpoints in $\{d_1,d_2,d_3\}$ are contained in a common connected component of $G[B]-v$. Now Lemma~\ref{lemma:K3minor} implies the existence of a $K_3$-minor in $G[B]$ rooted at $A \cap B$. Let $(D_i)_{i=1}^{3}$ be the branch-sets of such a rooted $K_3$-minor in $G[B]$ with $d_i \in D_i, i=1,2,3$. 

Moving on, fix any proper $4$-coloring $c:V(G) \rightarrow [4]$ of $G$. W.l.o.g. (possibly after relabelling colors and vertices) we may assume that one of the following three cases holds: (1) $c(d_1)=c(d_2)=c(d_3)=1$, (2) $c(d_1)=c(d_2)=1, c(d_3)=2$, (3) $c(d_1)=1, c(d_2)=2, c(d_3)=3$.

\medskip
\textbf{Case 1.} $c(d_1)=c(d_2)=c(d_3)=1$. Let $G_A$ be defined as the graph obtained from $G[A]$ by identifying $d_1, d_2, d_3$ into a single vertex $d_A$. Let $S_A \subseteq V(G_A)$ be defined as $S_A:=S$ if $S \cap \{d_1,d_2,d_3\}=\emptyset$ and $S_A:=S \cup \{d_A\}$ if $S \cap \{d_1,d_2,d_3\} \neq \emptyset$. Note that $c$ induces a proper $4$-coloring of $G_A$ by assigning color $1$ to $d_A$. Further, $G_A$ contains no $S_A$-rooted $K_4$-minor: Towards a contradiction, suppose that $(B_i)_{i=1}^{4}$ form branch-sets of an $S_A$-rooted $K_4$-minor in $G_A$. Then either $d_A \notin \bigcup_{i=1}^{4}{B_i}$ and then $(B_i)_{i=1}^{4}$ form an $S$-rooted $K_4$-minor in $G$, a contradiction. Otherwise, we have $d_A \in B_i$ for some $i \in [4]$ (w.l.o.g. $d_A \in B_1$) and then $((B_1\setminus \{d_A\}) \cup D_1 \cup D_2 \cup D_3, B_2, B_3, B_4)$ are easily seen to form the branch-sets of an $S$-rooted $K_4$-minor in $G$, again, a contradiction.
The facts that $\chi(G_A)\le 4$, that $G_A$ contains no $S_A$-rooted $K_4$-minor, that $|V(G_A)|<|V(G)|$ and the minimality assumption on $G$ now imply that $G_A$ admits a proper $4$-coloring $c_A:V(G_A)\rightarrow [4]$ such that no vertex in $S_A$ receives color $j$, for some $j \in [4]$. Possibly after permuting colors we may assume w.l.o.g. that $c_A(d_A)=1$. But then the coloring $c^\ast:V(G)\rightarrow [4]$ defined by $c^\ast(x):=c_A(x)$ for every $x \in A \setminus B$ and $c^\ast(x):=c(x)$ for every $x \in B$ forms a proper $4$-coloring of $G$ in which no vertex in $S$ receives color $j$, a contradiction to $S$ being colorful. This shows that Case~1 cannot occur. 

\medskip
\textbf{Case 2.}
$c(d_1)=c(d_2)=1, c(d_3)=2$. In this case, we define $G_A$ as the graph obtained from $G[A]$ by identifying only $d_1$ and $d_2$ into a new vertex $d_A$ and adding an edge between $d_A$ and $d_3$ (if it does not already exist). Again it is easy to see that the coloring $c$ induces a proper $4$-coloring of $G_A$ by assigning color $1$ to $d_A$. Let $S_A:=S$ if $S \cap \{d_1,d_2\}=\emptyset$ and $S_A:=S \cup \{d_A\}$ otherwise. We claim that there is no $S_A$-rooted $K_4$-minor in $G_A$. Towards a contradiction, suppose such a minor would be described by branch-sets $(B_i)_{i=1}^{4}$.
For each $i \in [4]$, set 
$$B_i^\ast:=\begin{cases}B_i, & \text{if }B_i \cap \{d_A,d_3\}=\emptyset, \cr (B_i\setminus \{d_A\}) \cup (D_1 \cup D_2), & \text{if }B_i \cap \{d_A,d_3\}=\{d_A\}, \cr (B_i\setminus \{d_3\}) \cup D_3, & \text{if }B_i \cap \{d_A,d_3\}=\{d_3\}, \cr (B_i\setminus \{d_A,d_3\}) \cup (D_1 \cup D_2 \cup D_3), & \text{if } B_i \supseteq \{d_A,d_3\}
\end{cases}.$$  It is then easy to see that $(B_i^\ast)_{i=1}^{4}$ form an $S$-rooted $K_4$-minor in $G$, a contradiction. 

Given that $G_A$ is a $4$-colorable graph with no $S_A$-rooted $K_4$-minor which is smaller than $G$, there must exist a proper $4$-coloring $c_A:V(G_A)\rightarrow [4]$ of $G_A$ such that for some $j \in [4]$ no vertex in $S_A$ receives color $j$. Since necessarily $c_A(d_A) \neq c_A(d_3)$ in this coloring, we may assume w.l.o.g. (possibly after permuting colors) that $c_A(d_A)=1, c_A(d_3)=2$. But then it is easy to see that $c^\ast:V(G) \rightarrow [4]$, defined by $c^\ast(x):=c_A(x)$ for every $x \in A\setminus B$ and $c^\ast(x):=c(x)$ for every $x \in B$ forms a proper coloring of $G$ such that no vertex in $S$ receives color $j$. This is a contradiction to $S$ being colorful and shows that also Case~2 cannot occur.  

\medskip
\textbf{Case 3.} $c(d_1)=1, c(d_2)=2, c(d_3)=3$. In this final case we define a graph $G_A$ on the vertex set $A$ simply by adding all the three edges $d_1d_2, d_1d_3,d_2d_3$ to $G[A]$ (if they do not exist already). Since $c$ induces a proper coloring also with the added edges, we have $\chi(G_A)\le 4$. We further claim that $G_A$ has no $S$-rooted $K_4$-minor. Indeed if it did, then we could easily obtain a $K_4$-minor rooted at $S$ in $G$ by replacing each occurrence of a vertex $d_i$ in one of its branch-sets by the full set $D_i\subseteq B$ in $G$. Hence, and by minimality of $G$, we find that there is a proper $[4]$-coloring $c_A$ of $G_A$ in which no vertex in $S$ receives color $j$, for some $j \in [4]$. Since $c_A(d_1),c_A(d_2),c_A(d_3)$ have to be pairwise different, we may assume w.l.o.g. $c_A(d_i)=i$ for $i=1,2,3$. However, then clearly the coloring $c^\ast$ of $G$ which is the common extension of $c_A$ and the restriction of $c$ to $B$ is a proper $[4]$-coloring of $G$ in which no vertex in $S$ receives color $j$. This again is a contradiction to our assumption that $S$ is colorful in $G$. Thus, Case~3 cannot occur as well.

As we have reached contradictions in all $3$ cases, we conclude that our initial assumption on the existence of a $3$-separation $(A,B)$ with $S \subseteq A$ was wrong. This proves the claim, $S$ is indeed spread out.
\end{claimproof}

Having established that $G$ is a $3$-connected graph and $S$ is spread out in $G$ in the previous claims, and since $G$ has no $S$-rooted $K_4$-minor by assumption, we now may apply Lemma~\ref{lma:rootyness} to $G$ (note that $|S|\ge 4$ holds trivally, for otherwise $S$ is not colorful). We thus find that $G^a(S)$ (the graph obtained from $G$ by adding a new vertex adjacent to all elements of $S$) is planar. But then the $4$-color theorem guarantees the existence of a proper $4$-coloring of $G^a(S)$. Then in the induced proper $4$-coloring of $G$, no vertex in $S$ can have the color assigned to the additional vertex in $G^a(S)$. This is a contradiction to $S$ being colorful. This contradiction shows that our initial assumption on the existence of a smallest counterexample $G$ was wrong, and concludes the proof of Theorem~\ref{thm:main}. 
\end{proof}

\end{document}